\documentclass[a4paper,10pt]{article}

\usepackage{amssymb}
\usepackage{amsmath}
\usepackage{amsfonts}

\newtheorem{theorem}{Theorem}
\newtheorem{corollary}{Corollary}

\newenvironment*{proof}
    {\begin{trivlist}\item[\hspace{\labelsep}\textit{Proof.}]}
    {\hspace*{\fill}\rule{0.5em}{0.5em}\end{trivlist}}

\begin{document}

\title{Asymptotics for numbers of line segments and
lines in a square grid}

\author{Pentti Haukkanen\footnote{E-mail:
pentti.haukkanen@uta.fi}\,\, and Jorma K.
Merikoski\footnote{E-mail: jorma.merikoski@uta.fi}\\ School
of Information Sciences\\ FI-33014 University of Tampere,
Finland}

\maketitle

\begin{abstract}

We present an asymptotic formula for the number of line
segments connecting $q+1$ points of an $n\times n$ square
grid, and a sharper formula, assuming the Riemann
hypothesis. We also present asymptotic formulas for the number
of lines through at least  $q$ points and, respectively,
through exactly $q$ points of the grid. The well-known case
$q=2$ is so generalized.



\end{abstract}

\section{Introduction}

Given $n\ge 2$, let us consider the
grid
$$
G(n)=
\{0,\dots,n-1\}\times\{0,\dots,n-1\}.
$$
Call its points {\it gridpoints}.
Given
$q\ge 2$, we say that a line is a {\it q-gridline} if it goes
through exactly $q$ gridpoints. We write
$l_q(n)$ for the number of $q$-gridlines,
and
$l_{\ge q}(n)$ for the number of gridlines
through at least $q$ gridpoints. In other words, $l_{\ge
q}(n)$ is the sum of all
$l_p(n)$'s with $p\ge q$.

\smallskip

We also say that a line segment is a {\it
q-gridsegment} if its endpoints and exactly $q-2$ interior
points are gridpoints. Let
$s_q(n)$ denote the number of all
$q$-gridsegments. (If $q>2$, some of them may partially
overlap.) In other words, $s_q(n)$ is the number
of line segments between gridpoints visible to each other
through $q-2$ gridpoints.

\smallskip

It is well-known~\cite{HM1,Mu1,Mu2,Sl} that, for all $n\ge 2$,
$q\ge 1$,
\begin{eqnarray}
\label{sqn}
s_{q+1}(n)=\frac{1}{2}f_q(n)
\end{eqnarray}
and, for all $q\ge 2$,
\begin{eqnarray}
\label{lgqn}
l_{\ge
q}(n)=\frac{1}{2}(f_{q-1}(n)-
f_q(n))
\end{eqnarray}
and
\begin{eqnarray}
\label{lqn}
l_q(n)=\frac{1}{2}(f_{q+1}(n)
-2f_q(n)+f_{q-1}(n)).
\end{eqnarray}
Here
\begin{eqnarray*}
\label{fqn}
f_q(n)=
\sum_{\substack{-n<i,j<n\\(i,j)=q}}(n-|i|)(n-|j|)
\end{eqnarray*}
%\end{theorem}
and $(i,j)$ denotes the greatest common divisor of $i$
and~$j$.

\smallskip

These ''explicit'' formulas may have theoretical value but
their practical value is small. They are computationally
tedious and do not tell much about the behaviour of the
functions. This motivates to look for recursive or asymptotic
formulas. For recursive formulas, see~\cite{HM1,Mu2} (and in
case of
$q=2$ also~\cite{EMHM,Mu1}). We study asymptotic formulas
here. The case of $q=2$ has already been
settled~\cite{EMHM}. More generally, corresponding asymptotic
formulas in an
$m\times n$ rectangular grid are obtained applying results
given in~\cite{HM2}, but they are weaker.

\smallskip

We will in Section~2 find an asymptotic formula for~$f_q(n)$
and, as corollaries, asymptotic formulas for
$s_{q+1}(n)$,
$l_{\ge q}(n)$ and~$l_q(n)$. In
Section~3, we will complete our paper with conclusions and
remarks. The Riemann hypothesis (RH) will have an
interesting role.

\section{Asymptotic formulas}

\begin{theorem}
Let $n,q\ge 1$. Then
\begin{eqnarray}
\label{fqnasymp}
f_q(n)=\frac{6n^4}{\pi^2q^2}+r(n),
\end{eqnarray}
where
\begin{eqnarray}
\label{orn}
r(n)=O(n^3\exp(-A(\log{n})^\frac{3}{5}
(\log\log{n})^{-\frac{1}{5}}))
\end{eqnarray}
for a certain~$A>0$. Assuming RH,
\begin{eqnarray}
\label{ornrh}
r(n)=O(n^{\frac{5}{2}+\varepsilon})
\end{eqnarray}
for all~$\varepsilon>0$.
\end{theorem}

In cases of $q=1,2$, this has already been proved~\cite{EMHM}.
The same idea works in the general case.

\begin{proof}
We divide the proof
into nine parts.

\medskip
\noindent
{\it 1. Evaluating} $f_q(n+1)$ {\it preliminarily}. Let
$n=qm+t$,
$0\le t<q$, and let
$\phi$ denote the Euler totient function.
By~\cite[Lemma~2]{EMHM},
\begin{eqnarray}
\label{f}
\nonumber
f_q(n+1)=8\sum_{i=1}^{\lfloor\frac{n}{q}\rfloor}(n+1-qi)
\big(n+1-\frac{q}{2}i\big)\phi(i)=\qquad\qquad\qquad\qquad\qquad
\\
\nonumber
8\sum_{i=1}^m(qm+t+1-qi)\big(qm+t+1-\frac{q}{2}i\big)\phi(i)=
\\
\nonumber
8\sum_{i=1}^m(qm-qi)\big(qm+t+1-\frac{q}{2}i\big)\phi(i)+\qquad
\qquad\qquad\qquad\qquad\qquad
\\
8(t+1)\sum_{i=1}^m\big(qm+t+1-\frac{q}{2}i\big)\phi(i)=
8s_1+8(t+1)s_2.
\end{eqnarray}

\medskip
\noindent
{\it 2. Evaluating} $s_1$ {\it preliminarily}. In the
partial summation formula
\begin{eqnarray}
\label{partsum}
\sum_{i=1}^Na_ib_i=\Big(\sum_{i=1}^Na_i\Big)b_N-
\sum_{i=1}^{N-1}\Big(\sum_{j=1}^ia_j\Big)(b_{i+1}-b_i),
\end{eqnarray}
substitute $N=m$, $a_i=\phi(i)$ and
$$
b_i=(qm-qi)\big(qm+t+1-\frac{q}{2}i\big);
$$
then
\begin{eqnarray*}
b_{i+1}-b_i=-\frac{3q^2}{2}m+q^2\big(i+\frac{1}{2}\big)-(t+1)q.
\end{eqnarray*}
Denoting
$$
\Phi(i)=\sum_{j=1}^i\phi(j)
$$
and
$$
c=(t+1)q-\frac{1}{2}q^2,
$$
we have
\begin{eqnarray}
\label{s1a}
s_1=\sum_{i=1}^{m-1}\Phi(i)\big(\frac{3q^2}{2}m-q^2i+c\big)=
\frac{3q^2}{2}m\sum_{i=1}^{m-1}\Phi(i)-
q^2\sum_{i=1}^{m-1}\Phi(i)i+c\sum_{i=1}^{m-1}\Phi(i).
\end{eqnarray}
Substituting $N=m-1$, $a_i=\Phi(i)$ and $b_i=i$
in~(\ref{partsum}) yields
$$
\sum_{i=1}^{m-1}\Phi(i)i=(m-1)\sum_{i=1}^{m-1}\Phi(i)-
\sum_{i=1}^{m-2}\sum_{j=1}^i\Phi(j),
$$
and so, continuing from~(\ref{s1a}),
\begin{eqnarray}
s_1=\big[\frac{3q^2}{2}m-q^2(m-1)+c\big]\sum_{i=1}^{m-1}\Phi(i)
+q^2\sum_{i=1}^{m-2}\sum_{j=1}^i\Phi(j)=\nonumber
\\
\big[\frac{q^2}{2}m+
\frac{q^2}{2}+(t+1)q\big]\sum_{i=1}^{m-1}\Phi(i)
+q^2\sum_{i=1}^{m-2}\sum_{j=1}^i\Phi(j).
\label{s1b}
\end{eqnarray}
Now, define
$$
E_\Phi(i)=\Phi(i)-\frac{3i^2}{\pi^2}.
$$
Then (\ref{s1b}) reads
\begin{eqnarray}
\label{s1c}
s_1=\big[\frac{q^2}{2}m+\frac{q^2}{2}+(t+1)q\big]
\sum_{i=1}^{m-1}\big(\frac{3i^2}{\pi^2}+E_\Phi(i)\big)
+q^2\sum_{i=1}^{m-2}\sum_{j=1}^i\big(\frac{3j^2}{\pi^2}
+E_\Phi(j)\big)\nonumber
\\
=s_{11}+s_{12}.
\end{eqnarray}

\medskip
\noindent
{\it 3. Evaluating} $s_{11}$. Let us define
$$
E_R(i)=\sum_{j=1}^iE_\Phi(j)-\frac{3i^2}{2\pi^2}.
$$
Then
\begin{eqnarray}
\label{s11}
s_{11}=\big[\frac{q^2}{2}m+\frac{q^2}{2}+(t+1)q\big]
\frac{3}{\pi^2}\frac{(m-1)m(2m-1)}{6}+\qquad\qquad\nonumber
\\
\big[\frac{q^2}{2}m+\frac{q^2}{2}+(t+1)q\big]
\big[\frac{3}{2\pi^2}(m-1)^2+E_R(m-1)\big]=\qquad\nonumber
\\
\frac{q^2}{2\pi^2}m^4+\frac{1}{2\pi^2}[q^2+2(t+1)q]m^3+O(m^2)+
O(m|E_R(m-1)|).
\end{eqnarray}

\medskip
\noindent
{\it 4. Evaluating} $s_{12}$. We have
\begin{eqnarray}
\label{s12}
\nonumber
s_{12}=q^2\sum_{i=1}^{m-2}
\Big(\frac{3}{\pi^2}\frac{i(i+1)(2i+1)}{6}+\frac{3}{2\pi^2}i^2
+E_R(i)\Big)=\qquad\qquad
\\
\nonumber
\frac{q^2}{\pi^2}\sum_{i=1}^{m-2}(i^3+3i^2+
\frac{1}{2}i)+q^2\sum_{i=1}^{m-2}E_R(i)=\qquad\qquad
\\
\nonumber
\frac{q^2}{\pi^2}\Big[\frac{(m-2)^2(m-1)^2}{4}+
3\frac{(m-2)(m-1)(2m-3)}{6}+O(m^2)\Big]+
\\
\nonumber
O(\sum_{i=1}^{m-2}|E_R(i)|)=\qquad
\\
\frac{q^2}{4\pi^2}m^4-\frac{q^2}{2\pi^2}m^3+O(m^2)+
O(\sum_{i=1}^{m-2}|E_R(i)|).
\end{eqnarray}

\medskip
\noindent
{\it 5. Evaluating} $s_1$ {\it finally.} By (\ref{s1c}),
(\ref{s11}) and (\ref{s12}),
\begin{eqnarray}
\label{s1d}
\nonumber
s_1=s_{11}+s_{12}=\frac{q^2}{2\pi^2}m^4+
\frac{1}{2\pi^2}[q^2+2(t+1)q]m^3+O(m^2)+\qquad\qquad\qquad
\\
\nonumber
O(m|E_R(m-1)|)
+\frac{q^2}{4\pi^2}m^4-\frac{q^2}{2\pi^2}m^3+O(m^2)+
O(\sum_{i=1}^{m-2}|E_R(i)|)=
\\
\frac{3q^2}{4\pi^2}m^4+\frac{q}{\pi^2}(t+1)m^3+O(m^2)+
O(m|E_R(m-1)|)+O(\sum_{i=1}^{m-2}|E_R(i)|).
\end{eqnarray}

\medskip
\noindent
{\it 6. Evaluating} $s_2$. Substituting
$N=m$, $a_i=\phi(i)$ and
$$
b_i=qm+t+1-\frac{q}{2}i
$$
in~(\ref{partsum}), we get
\begin{eqnarray}
\label{s2}
\nonumber
s_2=\Phi(m)\big(\frac{q}{2}m+t+1\big)-
\sum_{i=1}^{m-1}\Phi(i)\big(-\frac{q}{2}\big)=
\big(\frac{q}{2}m+t+1\big)\Phi(m)+
\frac{q}{2}\sum_{i=1}^{m-1}\Phi(i)
\\
\nonumber
=\big(\frac{q}{2}m+t+1\big)\big(\frac{3}{\pi^2}m^2+E_\Phi(m)\big)
+\frac{q}{2}\sum_{i=1}^{m-1}\big(\frac{3}{\pi^2}i^2+
E_\Phi(i)\big)=\qquad\qquad
\\
\nonumber
\frac{3q}{2\pi^2}m^3+O(m|E_\Phi(m)|)+O(m^2)+
\frac{q}{2}\frac{3}{\pi^2}\frac{(m-1)m(2m-1)}{6}+
\frac{q}{2}\sum_{i=1}^{m-1}E_\Phi(i)=
\\
\frac{2q}{\pi^2}m^3+\frac{q}{2}E_R(m-1)+O(m|E_\Phi(m)|)+
O(m^2).
\end{eqnarray}

\medskip
\noindent
{\it 7. Evaluating} $f_q(n+1)$ {\it finally.}
By~(\ref{f}), (\ref{s1d}) and (\ref{s2}),
\begin{eqnarray*}
f_q(n+1)=8s_1+8(t+1)s_2=\qquad\qquad\qquad\qquad\qquad\qquad\qquad\qquad
\\
\frac{6q^2}{\pi^2}m^4+\frac{8q}{\pi^2}(t+1)m^3+
O(m|E_R(m-1)|)+O(\sum_{i=1}^{m-2}|E_R(i)|)+O(m^2)+
\\
\frac{16q}{\pi^2}(t+1)m^3+4q(t+1)E_R(m-1)+O(m|E_\Phi(m)|)+
O(m^2)=\qquad
\\
\frac{6q^2}{\pi^2}m^4+\frac{24q}{\pi^2}(t+1)m^3+B(m)+O(m^2),
\end{eqnarray*}
where
\begin{eqnarray}
\label{bm}
B(m)=O(m|E_\Phi(m)|)+O(m|E_R(m-1)|)+
O(\sum_{i=1}^{m-2}|E_R(i)|).
\end{eqnarray}
Substituting
$$
m=\frac{n-t}{q},
$$
we further obtain
\begin{eqnarray}
\label{ff}
\nonumber
f_q(n+1)=\frac{6q^2}{\pi^2}\big(\frac{n-t}{q}\big)^4+
\frac{24q}{\pi^2}(t+1)\big(\frac{n-t}{q}\big)^3+
B(\frac{n-t}{q})+O(n^2)=
\\
\nonumber
\frac{6}{\pi^2q^2}n^4-\frac{24}{\pi^2q^2}tn^3+
\frac{24}{\pi^2q^2}(t+1)n^3+B(\frac{n-t}{q})+O(n^2)=\qquad
\\
\nonumber
\frac{6}{\pi^2q^2}n^4+\frac{24}{\pi^2q^2}n^3+O(n^2)+
B(\frac{n-t}{q})+O(n^2)=\qquad
\\
\frac{6}{\pi^2q^2}(n+1)^4+B(\frac{n-t}{q})+O(n^2).\qquad
\end{eqnarray}

\medskip
\noindent
{\it 8. Estimating} $B(m)$. We have~\cite[\S~I.21]{MSC}
\begin{eqnarray}
\label{Efi}
E_\Phi(m)=O(m\log{m})
\end{eqnarray}
and~\cite[Eq.~(1.11)]{SS}
\begin{eqnarray}
\label{ER}
E_R(m)=O(m^2\exp(-A(\log{m})^\frac{3}{5}
(\log\log{m})^{-\frac{1}{5}}))
\end{eqnarray}
for a certain~$A>0$. We also have
\begin{eqnarray}
\label{ERrh}
E_R(m)=O(m^{\frac{3}{2}+\varepsilon})
\end{eqnarray}
for all~$\varepsilon>0$ if~\cite[Eq.~(3.41)]{Su} (and only
if~\cite{Co}) RH holds. (For~(\ref{Efi}), see
also~\cite[Lemma~4]{EMHM}. For~(\ref{ER}) and~(\ref{ERrh}),
see also~\cite[Lemma~7]{EMHM}.)

\smallskip

By (\ref{bm}), (\ref{Efi}) and
(\ref{ER}),
\begin{eqnarray}
\label{obm}
\nonumber
B(m)=O(\max_{k\le
m}(m|E_\Phi(k)|+m|E_R(k)|))=\qquad\qquad
\qquad\qquad
\\
\nonumber
O(\max_{k\le m}(mk\log{k}+mk^2\exp(-A(\log{k})^\frac{3}{5}
(\log\log{k})^{-\frac{1}{5}})))=\qquad
\\
\nonumber
O(m^2\log{m}+m^3\exp(-A(\log{m})^\frac{3}{5}
(\log\log{m})^{-\frac{1}{5}})))=\qquad
\\
\nonumber
O(m^3\exp(-A(\log{m})^\frac{3}{5}
(\log\log{m})^{-\frac{1}{5}}))=\qquad
\\
O(n^3\exp(-A(\log{n})^\frac{3}{5}
(\log\log{n})^{-\frac{1}{5}})).\qquad
\end{eqnarray}
Assuming~RH, we proceed
similarly but use~(\ref{ERrh}) instead of~(\ref{ER}), so
obtaining
\begin{eqnarray}
\label{obmrh}
B(m)=O(n^{\frac{5}{2}+\varepsilon}).
\end{eqnarray}

\medskip
\noindent
{\it 9. Final conclusion.} Apply~(\ref{ff})
and~(\ref{obm}), and then replace $n£$ with~$n-1$. So
(\ref{fqnasymp}) with $r(n)$ in~(\ref{orn}) follows.
Under~RH, applying (\ref{obmrh}) instead of~(\ref{obm})
implies (\ref{fqnasymp}) with $r(n)$ in~(\ref{ornrh}).
\end{proof}

\begin{corollary}
Let $n,q\ge 1$. Then
$$
s_{q+1}(n)=\frac{3n^4}{\pi^2q^2}+r_1(n),
$$
where $r_1(n)$ has the $O$-estimates given in Theorem~1.
\end{corollary}
\begin{proof}
Apply~(\ref{sqn}).
\end{proof}

\begin{corollary}
Let $n\ge 1$, $q\ge 2$. Then
$$
l_{\ge q}(n)=\frac{3n^4}{\pi^2}\Big[\frac{1}{(q-1)^2}-
\frac{1}{q^2}\Big]+r_2(n),
$$
where $r_2(n)$ has the $O$-estimates given in Theorem~1.
\end{corollary}
\begin{proof}
Apply~(\ref{lgqn}).
\end{proof}

\begin{corollary}
Let $n\ge 1$, $q\ge 2$. Then
$$
l_q(n)=\frac{3n^4}{\pi^2}\Big[\frac{1}{(q+1)^2}-\frac{2}{q^2}+\frac{1}{(q-1)^2}\Big]+r_3(n),
$$
where $r_3(n)$ has the $O$-estimates given in Theorem~1.
\end{corollary}
\begin{proof}
Apply~(\ref{lqn}).
\end{proof}

\section{Conclusions and remarks}

Gerenalizing the asymptotic formulas presented in~\cite{EMHM},
we found asymptotic formulas for~$s_{q+1}(n)$, $l_{\ge q}(n)$
and~$l_q(n)$. It was crucial first to find an asymptotic
formula for~$f_q(n)$.

\smallskip

In all these formulas, the error term sharpens if we assume
RH. Interesting converse problems arise: Does
(\ref{ornrh}) imply~RH? If $r_1$ (respectively
$r_2$, $r_3$) satisfies~(\ref{ornrh}), does RH follow? There
are many equivalent statements of~RH,
see~\cite[Chapter~5]{BCRW} and~\cite{CNS}. Positive answers
to our questions would provide additional
elementary characterizations of~it.

\smallskip

A function~$d$, defined on~$G(n)$,
is a threshold function if it takes two
values 0 and~1 and if there is a line $a_1x_1+a_2x_2+b=0$
separating
$d^{-1}(\{0\})$ and $d^{-1}(\{1\})$ (i.e.,
$d(x_1,x_2)=0\Leftrightarrow a_1x_1+a_2x_2+b\le 0$). Let
$t(n)$ denote the number of such
functions. Alekseyev~\cite[Theorem~3]{Al} (see also \v
Zuni\'c~\cite{Zu}) proved (with different notation) that
$t(n)=f_1(n)+2$. So Theorem~1 also gives an asymptotic
formula for~$t(n)$.

\end{document}